\documentclass[a4paper,12pt,reqno]{amsart}
\usepackage[utf8]{inputenc}
\usepackage{amssymb,amsthm}
\usepackage{enumerate}
\usepackage{hyperref}
\usepackage[margin=3cm]{geometry}
\usepackage{graphicx}
\usepackage{color}

\renewcommand{\epsilon}{\varepsilon}

\newtheorem{theorem}{Theorem}[section]
\newtheorem{prop}[theorem]{Proposition}
\newtheorem{cor}[theorem]{Corollary}
\newtheorem{lemma}[theorem]{Lemma}
\newtheorem{remark}[theorem]{Remark}

\newcommand{\R}{\mathbb{R}}
\newcommand{\C}{\mathbb{C}}

\newcommand{\N}{\mathbb{N}}

\title[Solving polynomials with ODE]{Solving polynomials with\\ ordinary differential
equations}

\author[A. Gasull]{Armengol Gasull}
\address{Departament de Matem\`{a}tiques,
Universitat Aut\`{o}noma de Barcelona, Edifici C 08193 Bellaterra, Barcelona, Spain.} \email{gasull@mat.uab.cat}

\author[H. Giacomini]{Hector Giacomini}
\address{ Institut Denis Poisson. Universit\'{e} de Tours,  C.N.R.S. UMR 7013. 37200 Tours, France.} \email{Hector.Giacomini@lmpt.univ-tours.fr}

\date{}

\begin{document}

\begin{abstract} In this work we consider a given root of
a family of $n$-degree polynomials as a one-variable  function that
depends only on the independent term. Then we prove that this
function satisfies several  ordinary differential equations (ODE).
More concretely, it satisfies several simple separated variables
ODE, a first order generalized Abel ODE of degree $n-1$ and an
$(n-1)$-th order linear ODE. Although some of our results are not
new, our approach is simple and self-contained.
For $n=2, 3$ and $4$ we recover, from these ODE, the classical formulas for
solving these polynomials.
\end{abstract}

\maketitle

 \section{Introduction and main results}

It is known that although general polynomial equations of degree
$n\ge5$ can not be solved by radicals, their roots can be obtained
in terms of elliptic or hyperelliptic functions, their inverses or
other trascendental functions, like hypergeometric  or theta
functions. This is a classical subject which starts with results of
Hermite, Kronecker and  Brioschi and continues with contributions of
many others authors, see for instance \cite{bel,Di,Ki,PS} and the
references therein. We will not try to survey all the different points of
view from which the question of solving polynomials is addressed.

Because we usually work on ordinary differential equations (ODE) we
simply decided to explore which kind of results about polynomial
equations can be obtained by using ODE as a main tool. As we will
see, our results are self-contained and recover some of the known
results on the subject. Before we state our
contributions and compare them with these
known results, we present a brief survey of the most relevant
results that we have found on this subject that also use ODE as a
main tool.
To the best of our knowledge this approach started with the
contributions of Betti (\cite{be})  in 1854  and the ones around
1860 of Cockle and Hartley (\cite{coc,har}). In fact, Enrico Betti
proved that the solutions of general polynomial equations satisfy a
separated variables ODE and using this fact that he proved that the solutions of
these equations can be obtained in terms of hyperelliptic functions
and their inverses. He also proved that for quintic equations it
suffices to consider elliptic functions and their inverses. On the
other hand, James Cockle and Robert Harley showed explicit linear ODE
satisfied for a solution of an arbitrary trinomial polynomial equation in terms of its coefficients.
 For instance, they found a linear homogeneous
ODE of $4$-th order for a solution $x(q)$ of the quintic polynomial
equation  in the Bring-Jerrard form $x^5-x+q=0.$ These results are
presented and extended a litle in the 1865
Boole's book \cite[pp. 190--199]{boo}.
In his Thesis (``premi\`{e}re th\`{e}se''), published as a book in 1874 and as an article
 in 1875 (\cite{tan}), Tannery consider a more general question and proved that each branch
$y=y(x)$ of an algebraic curve $F(x,y)=0,$ of degree $n$ in $y$
satisfies an $n$-th order linear homogeneous ODE.
In 1887, Heymann (\cite{hey}) showed that a solution of a
trinomial equation with only one parameter satisfies a linear ODE and realizes
 for the first time that
this solution can be expressed as a hypergeometric function. Mellin
(\cite{mel}) around 1915 published the quoted paper and others about
the representation of the solutions of polynomial equations in terms
of multiple integrals, proving again that the solution of a trinomial
equations satisfy a linear ODE.
Owing to the complexity of explicitly
constructing  these ODE for
non-trinomial polynomial equations, some authors decided to use
also functions of several variables, and their corresponding partial
differential equations, to express the roots of arbitrary polynomial
equations. Bellardinelli, in his extensive review work of 1960 (\cite{bel}) presents this type of results and discuss also the previous works about ODE.

We want to stress that results like
the one of Tannery have practical applications when people deal with
some generating functions appearing in combinatorial problems, see
for instance the nice paper of Comtet(\cite{com}) and his classical
book on Combinatorics (\cite{com2}). There, the author gives also a
proof that the branches of algebraic curves satisfy linear ODE and
several applications of this fact.

In a few words, in this paper we will recover with independent
proofs some of the above results and give a few new ones. More
concretely, we will re-obtain the results of Betti as a particular case of our more general result and we will give a
simple and constructive procedure to obtain non homogeneous $(n-1)$-order linear ODE
satisfied by the solutions of general polynomial equations of degree
$n.$ During the way, we prove that these solutions satisfy also Abel type
polynomial differential equations of degree $n-1.$ We also apply all
our results for small $n$ and to some particular examples. It is
funny to observe for instance that we obtain the celebrated
Cardano's formula for $n=3,$ by reducing the computations to solve
the simple second order ODE of the harmonic oscillator $x''=-x.$

Before stating our main results, we start giving an idea of our
approach to the problem. In principle, we do not consider all real
polynomials, but generic ones. Let $R(x)$ be a monic degree $n$ real
polynomial satisfying $R(0)=0$ and $R'(0)\ne0.$ We are interested
into the degree $n$ polynomial equation
\begin{equation}\label{eq:main}
P(x)=R(x)-q=0,
\end{equation}
$q\in\R,$ and more in particular, in finding a local explicit
expression of the analytic solution $x(q)$ of \eqref{eq:main} such
that $\lim_{q\to0} x(q)=0,$ which exists and is unique by the
implicit function theorem because $R'(0)\ne0.$

Let $F$ be any invertible diffeomorphism such that $F(0)=0.$ Then,
from \eqref{eq:main} we have the locally equivalent equation $F(
R(x))= F (q).$ We know that $F ( R( x(q)) )=F(q).$ By derivating it
with respect  to $q,$ and multiplying both sides by $G(R(x(q))=
G(q),$ with $G$ an arbitrary continuous function, we obtain an ODE
with separated variables for $x(q)$
\[
G(R(x(q))F'( R( x(q))) R' (x(q)) x'(q) = G(q) F'(q),
\]
with initial condition $x(0)=0.$ It can be easily solved giving rise
to
\begin{equation}\label{eq:edo}
\phi(x):=\int_0^x G(R(s))F'( R( s)) R' (s)\, {\rm d} s= \int_0^q
G(t) F'(t)\, {\rm d} t=:\varphi(q).
\end{equation}
Hopefully, this new equation $\phi(x)=\varphi(q),$ which is
equivalent to $R(x)=q,$ allows to obtain $x=\phi^{-1}(\varphi(q))$
being $\phi^{-1}\circ \varphi$ an explicit computable function that
it is not trivially equivalent to $x=R^{-1}(q).$ A key point is to
chose a suitable $F$ that provides {\bf some cancellation} in the
expression $F'( R( x) ) R' (x).$

As we will see, this wished cancellation  can be obtained from the
following result. Here $\operatorname{dis}_x (P(x))$ denotes the
discriminant of $P(x)$ with respect to $x,$ see \cite{St}.

\begin{prop}\label{propo}
Let $P(x)=R(x)-q$ be a real polynomial of degree $n\ge2,$ with
$R(0)=0$ and $q\in\R.$  Set $D(q)=\operatorname{dis}_x (P(x)).$ Then
\begin{equation}\label{eq:dec}
D(R(x))= (R'(x))^2 U(x)= (P'(x))^2 U(x),
\end{equation}
for some polynomial $U$ of degree $(n-1)(n-2).$ Moreover, if all the
roots of $R$ are simple, $U(0)\ne0.$
\end{prop}

Next theorem is our first main result and proves that a root of a
generic polynomial equations $P(x)=R(x)-q=0,$ for $q$ in a
neighborhood of $0,$ can be obtained in terms of hyperelliptic
functions and their inverses. As we will comment in
Remark~\ref{re:remark}, the restriction that all the roots of $R$
are simple can be removed obtaining a similar result. As we have
commented, this result is similar, but more general, to the one
given by Betti.

\begin{theorem}\label{teo}  Let $P(x)=R(x)-q$ be a real polynomial of degree $n\ge2,$ with $R(0)=0$ and $q\in\R.$
Set $D(q)=\operatorname{dis}_x (P(x))$ and assume that all the roots
of $R$ are simple. Define the polynomials $\mathcal{D}(q)=
\operatorname{sgn}(D(0)) D(q)$ and
$\mathcal{U}(x)=\mathcal{D}(R(x))/(R'(x))^2,$ and the functions
\begin{equation}\label{eq:fis}
\phi (x)= \operatorname{sgn}(R'(0))\int_0^x
\dfrac{G(R(s))}{\sqrt{\mathcal{U}(s)\,}}\,{\rm d} s
\quad\mbox{and}\quad \varphi (q)= \int_0^q
\dfrac{G(t)}{\sqrt{\mathcal{D}(t)\,}}\,{\rm d} t,
\end{equation}
where $G$ is any continuous function satisfying $G(0)\ne0.$ Then, in
a neighborhood of~$0,$  $\phi$ is invertible and
\[
x=\phi^{-1}(\varphi(q))
\]
is a root of $P(x)=0$ that goes to $0$ as $q$ tends to $0.$

In particular, if $G$ is polynomial, $\phi$ and $\varphi$ are
elliptic or hyperelliptic integrals.
\end{theorem}

As an illustration, we apply our results to the low degree cases. In
particular, when $n=2$ and $n=3$ we reobtain the Babylonian and
Cardano's formulas, see Sections \ref{se:bab} and \ref{se:car},
respectively. In Section \ref{se:quartica} we apply them to the
quartic case. Finally, in Section \ref{se:quintic} we reproduce the
results of Betti's work for quintic equations with our point of
view.

To state our second main result we recall some definitions. Given
$0<m\in\mathbb{N}$ we will say that a non autonomous first order real ODE of the
from
\begin{equation}\label{eq:abel}
x'=a_m(q)x^{m}+a_{m-1}(q)x^{m-1}+\cdots+a_2(q)x^2+ a_1(q)x+a_0(q)
\end{equation} is a {\it generalized Abel ODE of degree $m.$}
Notice that for $m=1,2$ and $3$ these equations are usually called
{\it linear, Riccati} and {\it Abel} ODE, respectively. All of them
are a subject of classical interest in mathematics.

By using Corollary \ref{coro3},  we prove:

\begin{theorem}\label{teo2}  Let $P(x)=R(x)-q$ be a real polynomial of degree $n\ge2,$ with $R(0)=0$ and
$q\in\R.$ Let $x(q)$ be one of the roots of this equation, defined
in a neighborhood of $0,$ that tends to zero as $q$ tends to $0.$
Then $x(q)$ satisfies a generalized Abel ODE \eqref{eq:abel} of
degree $m=n-1,$ where  $a_j(q),j=0,1,\ldots,n-1$ are rational
functions with coefficients depending on the coefficients of $R.$
\end{theorem}

A straightforward consequence of this result is the following corollary.

\begin{cor}  Let $P(x)=R(x)-q$ be a real  quadratic, cubic or quartic polynomial equation with $R(0)=0.$
Let $x(q)$ be one of the roots of this equation, defined in a
neighborhood of $0,$ that tends to zero as $q$ tends to $0.$ Then
$x(q)$ satisfies, respectively,  a linear, Riccati or Abel ODE whose
coefficients are rational functions in $q.$
\end{cor}

The proof of the above results, together with the explicit ODE when
$n\in\{2,3,4\}$ and $P$ has the canonical form $P(x)=x^n+px-q$ are
given in Section \ref{se:abel}.

A second consequence of the above results is:

\begin{theorem}\label{teo3}  Let $P(x)=R(x)-q$ be a real polynomial of degree $n\ge2,$ with $R(0)=0$ and
$q\in\R.$ Let $x(q)$ be one of the roots of this equation, defined
in a neighborhood of $0,$ that tends to zero as $q$ tends to $0.$
Then $x=x(q)$ satisfies a $(n-1)$-th order linear ODE,
\[
b_{n-1}(q)x^{(n-1)}+b_{n-2}(q)x^{(n-2)}+\cdots+b_{1}(q)x'+b_{0}(q)x+b_{n}(q)=0,
\]
where the functions $b_j(q)$ are polynomials in  $q,$  with
coefficients depending on the coefficients of $R.$
\end{theorem}

Our proof provides a constructive
algorithm to obtain all the functions $b_j$. As we will see in
Section \ref{se:order-n}, when the equation $P(x)=0$ is the the
trinomial one, $P(x)=x^n+px-q=0,$ these $b_j$ are extremely simple.
We obtain them for $n\le6.$ We also recover again Cardano's formula
by showing that for $n=3,$ and with suitable changes of variables,
this differential equation can be written as the equation for the harmonic
oscillator. For $n=4$ in Section \ref{se:444} we present three
different expressions of its solution $x(q),$ two of them in terms
of hypergeometric functions, and also the classical one.

Finally, in a short Appendix, we present some classical ways to
solve the cubic and the quartic equations. This is done, not only
for completeness, but for obtaining a simple and suitable way (for
our interests) of presenting the solution of the quartic equation.

\section{Proof of Theorem \ref{teo} and some applications}

We start proving Proposition \ref{propo}.

\begin{proof}[Proof of Proposition \ref{propo}] The polynomial $P'=R'$ has degree $n-1.$ We give the proof
when all its roots in $\C$, $\alpha_1,\alpha_2,\ldots,\alpha_{n-1},$
are different and moreover $R(\alpha_j)\ne R(\alpha_k)$ unless
$j=k.$ The general result follows from this generic case.

Consider the system
\begin{equation}\label{eq:com}
\begin{cases}
&P(x)=R(x)-q=0,\\
&P'(x)=R'(x)=0.
\end{cases}
\end{equation}
Recall, that modulus some non-zero constant, the discriminant
between $P$ and $P'$ is the resultant. Hence, by the properties of
the resultant we know that $D(q)$ is a polynomial of degree $n-1$
and moreover it vanishes for all values of $q$ for which the above
system is compatible.

Thus, system \eqref{eq:com} is compatible if and only if
$q=R(\alpha_j),$ $j=1,2,\ldots, n-1.$ Hence
\[
D(q)=K \big(q-R(\alpha_1) \big)\big(q-R(\alpha_2) \big)\cdots
\big(q-R(\alpha_{n-1}) \big),\quad K\ne0.
\]
Consider the new polynomial $Q(x)=D(R(x)),$ of degree $n(n-1).$ It
is clear that $Q(\alpha_j)=D(R(\alpha_j))=0.$ Moreover,  $Q'(x)=
D'(R(x)) R'(x).$ Therefore, $Q'(\alpha_j)= D'(R(\alpha_j))
R'(\alpha_j)=0.$ As a consequence, all $\alpha_j$ are double roots
of $Q(x)$ and \eqref{eq:dec} holds.

Finally, since $D(0)=C\operatorname{Res}_x(R(x),R'(x)),$ with
$C\ne0,$ the hypothesis that all the roots of $R$ are simple is
equivalent to $D(0)\ne0.$ Since $R'(0)\ne0,$ we get that $U(0)\ne0,$
as we wanted to prove.
\end{proof}

\begin{proof}[Proof of Theorem \ref{teo}] Our proof starts with the discussion  given  in the
introduction of the paper and uses the notations introduced there.
Recall that \eqref{eq:edo} writes as
\begin{equation}\label{eq:edo2}
\phi(x)=\int_0^x G(R(s))F'( R( s)) R' (s)\, {\rm d} s= \int_0^q G(t)
F'(t)\, {\rm d} t=\varphi(q).
\end{equation}
We take
\[
F(t)=\int_0^t\dfrac1{\sqrt{\mathcal{D}(s)\,}}\, {\rm d} s,
\]
that is a local diffeomorphism at $0$ because $\mathcal{D}(0)=
\operatorname{sgn}(D(0))D(0)>0$ and, as a consequence,
$F'(0)=1/\sqrt{\mathcal{D}(0)\,}\ne0.$ Then, by Proposition
\ref{propo},
\begin{equation}\label{eq:relacio}
F'(R(s))=\dfrac1{\sqrt{\mathcal{D}(R(s))\,}}=\dfrac1{\sqrt{(R'(s))^2\mathcal{U}(s)\,}}
=\dfrac{\operatorname{sgn}(R'(0))}{R'(s)\sqrt{\mathcal{U}(s)\,}}
\end{equation}
and \eqref{eq:edo2} reads as
\[\phi (x)=
\operatorname{sgn}(R'(0))\int_0^x
\dfrac{G(R(s))}{\sqrt{\mathcal{U}(s)\,}}\,{\rm d} s = \int_0^q
\dfrac{G(t)}{\sqrt{\mathcal{D}(t)}\,}\,{\rm d} t=\varphi(q).
\]
Notice that $\varphi'(0)=G(0)/\sqrt{\mathcal{D}(0)\,}\ne0.$ Hence,
$\varphi$ is a local diffeomorphism at $0.$ The same happens with
$\phi,$ because $\phi'(0)={\operatorname{sgn}(R'(0))
G(0)}/{\sqrt{\mathcal{U}(0)\,}}\ne0.$ Thus $x=\phi^{-1}(\varphi(q))$
as we wanted to prove.
\end{proof}

Next remark clarifies the situation when some of the hypotheses
Theorem \ref{teo} are not satisfied.

\begin{remark}\label{re:remark} In Theorem \ref{teo} the hypothesis that all the
roots of $R$ are simple is used to ensure that $R'(0)\ne0$ and
$D(0)\ne0.$ These conditions together with the hypothesis that
$G(0)\ne0$ imply that the functions defined by the hyperelliptic integrals
 $\phi$ and $\varphi$ are invertible at
$0.$ If we do not mind about their invertibility we arrive also to
equality $\phi(x)=\varphi(q)$ with these functions given as in
\eqref{eq:fis} of the statement. We also remark that in this
situation the value $\operatorname{sgn}(R'(0))$  must be replaced by
the sign of $R(s)$ for $s$ in the interval containing $0$ and $x.$

Moreover, in this situation, if $R'(0)=0,$ that is when near $x=0$
the polynomial equation $R(x)=q$ writes as $x^k+O(x^{k+1})=q,$ for
some $1<k\in\N,$ $k$ smooth branches of solutions, $x_j(q),
j=1,\ldots,k,$ solve the equation and satisfy $x_j(0)=0.$ This
result is a consequence of Weierstrass' preparation theorem. Each one
of these branches satisfies the ODE that we are considering and, as
a consequence,   the equality $\phi(x)=\varphi(q)$ with both
functions given in  \eqref{eq:fis}. Similar branches appear also
when we try to invert $\phi.$
\end{remark}

We will also need the following corollaries of previous results.
Notice that from the first corollary, the functions defined by the
hyperelliptic integrals given in Theorem~\ref{teo} are replaced by
primitives of rational functions.

\begin{cor}\label{coro2}  Let $P(x)=R(x)-q$ be a real polynomial of degree $n\ge2,$ with $R(0)=0$ and $q\in\R.$
    Set $D(q)=\operatorname{dis}_x (P(x))$ and assume that all the roots
    of $R$ are simple. Define the polynomial
    $U(x)={D}(R(x))/(R'(x))^2$  and the functions
    \begin{equation*}
    \Phi (x)= \int_0^x
    \dfrac{H(R(s))}{R'(s){{U}(s)}}\,{\rm d} s
    \quad\mbox{and}\quad \Psi (q)= \int_0^q
    \dfrac{H(t)}{{{D}(t)}}\,{\rm d} t,
    \end{equation*}
    where $H$ is any continuous function satisfying $H(0)\ne0.$
    Then,  in a neighborhood of $0,$  $\Phi$ is
invertible and
    \[
    x=\Phi^{-1}(\Psi(q))
    \]
    is a root of $P(x)=0$ that goes to $0$ as $q$ tends to $0.$

    In particular, if $H$ is polynomial, $\Phi$ and $\Psi$ are
primitives of rational functions.
\end{cor}

\begin{proof}
To prove this result we take
\[
G(t)=\frac{H(t)}{\sqrt{\mathcal{D}(t)\,}}
\]
in Theorem \ref{teo}. Notice that by using  \eqref{eq:relacio} we
obtain that
\[
G(R(s))=\dfrac{H(R(s))}{\sqrt{\mathcal{D}(R(s))\,}}=\dfrac{\operatorname{sgn}(R'(0))
H(R(s))}{R'(s)\sqrt{\mathcal{U}(s)\,}}.
\]
Therefore,
\[
\operatorname{sgn}(R'(0))
\dfrac{G(R(s))}{\sqrt{\mathcal{U}(s)\,}}=\big(\operatorname{sgn}(R'(0))\big)^2
\dfrac{
H(R(s))}{R'(s)\big(\sqrt{\mathcal{U}(s)\,}\big)^2}=\operatorname{sgn}(U(0))\dfrac{
H(R(s))}{R'(s)U(s)}.
\]
Similarly,
\[
\dfrac{G(t)}{\sqrt{\mathcal{D}(t)\,}}=\dfrac{H(t)}{\big(\sqrt{\mathcal{D}(t)\,}\big)^2}=\operatorname{sgn}(D(0))\dfrac{H(t)}{D(t)}
=\operatorname{sgn}(U(0))\dfrac{H(t)}{D(t)}.
\]
By replacing both expressions in \eqref{eq:fis} we obtain that
$\Phi(x)=\Psi(q)$ and the corollary follows.
\end{proof}

This second corollary is essentially a version of  Remark
\ref{re:remark} in this situation. Notice that the hypothesis that
all the roots of $R$ are simple it is not needed.

\begin{cor}\label{coro3}  Let $P(x)=R(x)-q$ be a real polynomial of degree $n\ge2,$ with $R(0)=0$ and $q\in\R.$
    Set $D(q)=\operatorname{dis}_x (P(x)).$  Let $x=x(q)$ be  a root of $P(x)=0$ that goes to $0$ as $q$ tends to
    $0.$ Then
\[
x'=\frac{R'(x) U(x)}{D(q)},
\]
where $U$ is the polynomial   $U(x)={D}(R(x))/(R'(x))^2.$
\end{cor}

\begin{proof} By Weierstrass' Preparation theorem we know that the algebraic curve $P(x)=R(x)-q$ has at
most $n$ branches passing by the point $(x,q)=(0,0).$ Moreover, each
of these branches, say $x=x(q),$ satisfies $R(x(q))=q.$ Hence,
$R'(x(q))x'(q)=1.$ From Proposition \ref{propo}, it holds that
$D(R(x))= (R'(x))^2 U(x)$ and, as a consequence,
\[
x'=\frac1{R'(x)}=\frac{R'(x) U(x)}{D(R(x))}=\frac{R'(x) U(x)}{D(q)},
\]
as desired.
\end{proof}

We will apply the above results for $n\le5.$

\subsection{A toy example: the quadratic equation}\label{se:bab}

Consider $P(x)=x^2+px-q,$ with $p\ne0.$ Then
$D(q)=\operatorname{dis}_x (P(x))=p^2+4q,$
 $\mathcal{D}(q)\equiv D(q)$ and $\mathcal{D}(R(x))= (2x+p)^2.$ Then $\mathcal{U}=1.$
  Moreover, since $R'(0)=p,$ we get  from~\eqref{eq:fis}, that for $|q|<p^2/4,$
\begin{align*}
\phi (x)&= \int_0^x
\dfrac{\operatorname{sgn}(R'(0))}{\sqrt{\mathcal{U}(s)\,}}\,{\rm d}
s= \int_0^x \operatorname{sgn}(p)
\,{\rm d} s= \operatorname{sgn}(p) x,\\
\varphi (q)&= \int_0^q \dfrac1{\sqrt{\mathcal{D}(t)\,}}\,{\rm d}
t=\int_0^q \dfrac1{\sqrt{p^2+4t\,}}\,{\rm d} t=\left.\frac12
\sqrt{p^2+4t\,}\right|_0^q= \frac{ \sqrt{p^2+4q\,}-\sqrt{p^2\,}}{2}.
\end{align*}
Then, by Theorem~\ref{teo} we get equation $\phi(x)=\varphi(q),$
that gives the Babylonian formula
\[
x= \frac{-p+ \operatorname{sgn}(p)\sqrt{p^2+4q\,}}{2}.
\]

By using Corollary~\ref{coro2} instead of Theorem~\ref{teo}  with
$H=1$ we obtain
\begin{align*}
\Phi (x)&=\int_0^x \dfrac{1}{R'(s){{U}(s)}}\,{\rm d} s= \int_0^x
\dfrac{1}{2s+p}\,{\rm d} s= \frac12\log\left(\frac{2x+p}{p}\right),\\
\Psi (q)&=  \int_0^q \dfrac1{{{D}(t)}}\,{\rm d} t=\int_0^q
\dfrac1{p^2+4t}\,{\rm d} t=
\frac14\log\left(\frac{p^2+4q}{p^2}\right).
\end{align*}
By using that $\Phi(x)=\Psi(q)$ we obtain again the classical
formula.

Finally, notice that although the obtained  formula for $x=x(q)$  is
valid when $|q|<p^2/4,$ their algebraic nature makes it valid for
all values of $p$ and $q.$

\subsection{Cubic equations}\label{se:car}

We find a solution for the cubic polynomial equation
\begin{equation}\label{eq:cubicc}
P(x)=x^3+px-q=0.
\end{equation}
Notice the minus sign in front of $q,$ in contrast with the usual
notation given in~\eqref{eq:cubica} utilized in Section~\ref{ss:cla}
of the Appendix. We exclude the trivial case $p\ne0.$

In the notation of Theorem~\ref{teo}, $D(q)=-(4p^3+27q^2)$ and
 $\mathcal{D}(q)=\operatorname{sgn}(p)(4p^3+27q^2).$ After some computations,
\[
\mathcal{D}(R(x))=\operatorname{sgn}(p)(3x^2+4p)(3x^2+p)^2,
\]
and, as a consequence, $\mathcal{U}(x)=
\operatorname{sgn}(p)(3x^2+4p).$ Hence,  taking $G=1,$
 equation $\phi(x)=\varphi(q)$ writes as
 \begin{equation}\label{eq:ccc}
\int_0^x
 \dfrac{\operatorname{sgn}(p)}{\sqrt{\operatorname{sgn}(p)(3s^2+4p)\,}}\,{\rm d} s
 = \int_0^q
 \dfrac1{\sqrt{\operatorname{sgn}(p) (4p^3+27t^2)\,}}\,{\rm d} t.
 \end{equation}
It is well-known that
\begin{equation}\label{eq:int}
\int_0^x \frac{1}{\sqrt{Ay^2+B}}\,{\rm d} y=\begin{cases}
\dfrac{\operatorname{arcsinh}\left(\sqrt{A/B\,}
x\right)}{\sqrt{A\,}},
&\mbox{when} \quad A>0,\, B>0, \\[0.4cm]
\dfrac{\operatorname{arcsin}\left(\sqrt{-A/B\,}
x\right)}{\sqrt{-A\,}}, &\mbox{when} \quad A<0,\, B>0,
\end{cases}
\end{equation}
where the second equality is only valid for  for
$|x|<\sqrt{-B/A\,}.$ Thus, for instance applying the
 first one when $p>0$ in~\eqref{eq:ccc} we obtain that
\[
\frac{\sqrt 3}{3} \operatorname{arcsinh}\left(\frac{\sqrt3
x}{2\sqrt{p}}\right)=
\frac{\sqrt3}9\operatorname{arcsinh}\left(\frac32\frac{\sqrt3
q}{p\sqrt{p}}\right),
\]
or equivalently,
\begin{equation}\label{eq:solu}
x=\frac{2\sqrt{p}}{\sqrt{3}}\operatorname{sinh}\left(
\frac13\operatorname{arcsinh} \left(\frac32\frac{\sqrt3
q}{p\sqrt{p}}\right) \right).
\end{equation}
By using that $\operatorname{arcsinh}(z)=\ln
\left(z+\sqrt{z^2+1\,}\right)$ we obtain that
\[
\operatorname{sinh}\left(
\frac13\operatorname{arcsinh}\left(z\right) \right)=\frac12\left(
\sqrt[3]{z+\sqrt{z^2+1\,}\,}-\dfrac1{ \sqrt[3]{z+\sqrt{z^2+1\,}\,}}
\right)
\]
and hence, after some computations, from~\eqref{eq:solu} we get
\begin{equation}\label{eq:car2}
x=\root{3}\of{\frac{q}{2}+\sqrt{\frac{q^2}{4}+\frac{p^3}{27}\,}\,}
-\dfrac{p}{3\root{3}\of{\frac{q}{2}+\sqrt{\frac{q^2}{4}+\frac{p^3}{27}\,}}\,},
\end{equation}
that is, Cardano's formula for equation~\eqref{eq:cubicc}.

If we consider the case $p<0$ and perform the same type of
computations but
 using the second equality in~\eqref{eq:int} we arrive to
\begin{equation}\label{eq:solu22}
x=\frac{2\sqrt{-p\,}}{\sqrt{3}}\sin\left(
\frac13\arcsin\left(\frac32\frac{\sqrt3 q\,}{p\sqrt{-p\,}}\right)
\right),
\end{equation}
that is similar to~\eqref{eq:solu}, but for $p<0$ and only valid
when $|q|<\sqrt{-4p^3/27\,}.$

In any case, as for the quadratic equations, the algebraic nature of
the formula~\eqref{eq:car2} allows to consider it for all values of
$p$ and $q.$

\subsection{Quartic equations}\label{se:quartica}
 As we will see, it is difficult to
recover the classical solution with this approach. We start with a
particularly simple case. We will return to this case in Section
\ref{se:444}.

Consider the particular quartic equation
\begin{equation}\label{eq:q-partic}
P(x)=x^4-2x^3+2x^2-x-q=0.
\end{equation}
We will apply Theorem \ref{teo} with $G=-2.$ After some calculations
we obtain that
\[
\mathcal{U}(x)=(2x^2-2x+1)^2(4x^2-4x+3)\quad\mbox{and}\quad
D(q)=(4q+1)^2(16q+3).
\]
Hence,
\begin{align*}
\phi (x)=& \operatorname{sgn}(R'(0))\int_0^x
\dfrac{G(R(s))}{\sqrt{\mathcal{U}(s)\,}}\,{\rm d} s=
\int_0^x \frac2{(2s^2-2s+1)\sqrt{4s^2-4s+3\,}}\,{\rm d} s\\
=& 2\arctan \left(\frac{2x-1}{\sqrt{4x^2-4x+3\,}}\right)+\frac\pi3
\end{align*}
and
\begin{align*}
 \varphi (q)=& \int_0^q \dfrac{G(t)}{\sqrt{\mathcal{D}(t)\,}}\,{\rm
d} t=\int_0^q\frac{-2}{(4q+1)\sqrt{16q+3\,}} \,{\rm d} t=
-\arctan\left(\sqrt{16q+3\,}\right)+\frac\pi3.
\end{align*}
For the sake of shortness we introduce the new variables
\[
z=\frac{2x-1}{\sqrt{4x^2-4x+3\,}}\quad\mbox{and}\quad
w=\sqrt{16q+3\,}.
\]
Notice that given $z$ the corresponding values of $x$ can be
obtained by solving a quadratic equation. Hence, the equation
$\phi(x)=\varphi(q)$ can be written as
\[
2\arctan \left(z\right)= -\arctan\left(w\right),
\]
or, equivalently, $\tan(2\arctan \left(z\right))=-w,$ that gives
\[
\frac{2z}{z^2-1}=w.
\]
Thus, for each $w$, the corresponding value of $z$ can be obtained
again by solving a new quadratic equation $wz^2-w-2z=0.$

In short, solving two quadratic equations the quartic equation
\eqref{eq:q-partic} can be solved. In fact, this is the
particularity of the equation that we have considered and makes its
study easier: there is no need to solve any cubic equation to find
its roots. Their four solutions are
\[\frac12\pm\frac12\sqrt{-1\pm2\sqrt{1+4q\,}\,}.\]

Let us explore what gives our approach when we apply it to a general
quartic equation. Recall first, that similarly of what happens with
cubic equations, the general quartic case  can be reduced to
\begin{equation}\label{eq:quartica} x^4+px-q=0,
\end{equation}
for some $p,q\in\R.$ In this situation, a translation is not enough
to arrive to \eqref{eq:quartica} and the so-called   Tschirnhausen
transformations must be used.

If we apply Theorem~\ref{teo} with $G=1,$  we obtain that
$D(q)=-(27p^4+256 q^3)$ and $\mathcal{D}(q)=27p^4+256 q^3.$ Then,
some computations give
\[
\mathcal{D}(R(x))=(R'(x))^2 \mathcal{U}(x)=(4x^3+p)^2\big(16x^6+40 p
x^3+ 27 p^2\big).
\]
Hence,
\[
\phi (x)= \int_0^x
\dfrac{\operatorname{sgn}(R'(0))}{\sqrt{\mathcal{U}(s)\,}}\,{\rm d}
s= \int_0^x \dfrac{1}{\sqrt{16s^6+40 p s^3+ 27 p^2\,}}\,{\rm d} s
\]
and
\[
\varphi (q)= \int_0^q \dfrac{1}{\sqrt{\mathcal{D}(s)\,}}\,{\rm d} s=
\int_0^q \dfrac{1}{\sqrt{256 s^3+27p^4\,}}\,{\rm d} s.
\]
The above functions can be expressed as an Appell function and a hypergeometric
function, respectively. Therefore, this approach gives no
satisfactory results in order to obtain the roots of the quartic equation in terms of radicals.
We will return to the quartic equation in
Section \ref{se:444}.

\subsection{Quintic equations}\label{se:quintic}

In this section, with our approach, we recover the result of Betti
(\cite{be}) that asserts that the solution of these equations can be
obtained in terms of the inverse of an elliptic integral. Following
Betti, it suffices to study the particular quintic equation
\[
P(x)=x^5+5x^3-q.
\]
We can not apply directly Theorem \ref{teo} because the above case
is not under its hypotheses. In fact $R(x)=x^5+5x^3$ and hence
$R'(0)=0.$ Moreover, as we will see, we will use $G(x)={5\sqrt5}x$ and thus
$G(0)=0$. Therefore two of the hypotheses of the theorem are not
satisfied, but instead we will use the extended result explained in
Remark \ref{re:remark}. There we explain that in this more general
situation, it holds that $\phi(x)=\varphi(q)$ with $\phi$ and
$\varphi$ given also in \eqref{eq:fis}.

Following the notation of Theorem \ref{teo} we have that
\[
\mathcal{D}(q)=5^5q^2(q^2+108)\quad \mbox{and}\quad
\mathcal{U}(x)=5^3x^2(x^2+5)^2(x^6+4x^4-8x^2+12).
\]
Moreover, as it is explained in Remark \ref{re:remark},
$\operatorname{sign}(R'(0))$ can be replaced by $+1,$ because near
$0$, $R'$ is positive.

Taking $G(x)={5\sqrt5} x,$ we get that
\[
\phi(x)=\int_0^x \frac{s^2}{\sqrt{s^6+4s^4-8s^2+12\,}}\, {\rm d}
s\quad\mbox{and}\quad \varphi(q)=\int_0^q\frac{1}{5\sqrt{t^2+108\,}}
\, {\rm d} t.
\]
Hence, by introducing the new variable $u=s^2,$ the equality
$\phi(x)=\varphi(q)$ writes as
\[
\int_0^{x^2} \frac{u}{\sqrt{u(u^3+4u^2-8u+12)\,}}\, {\rm d}
u=\int_0^q\frac{2}{5\sqrt{t^2+108\,}} \, {\rm d} t.
\]
This expression is precisely the one obtained in \cite{be} and gives
a root of the considered quintic equation in terms of elementary
functions and the inverse of an elliptic integral.

\section{Polynomials and  Abel equations: proof of Theorem \ref{teo2}}\label{se:abel}

This section is devoted to prove Theorem \ref{teo2} and its
corollary.

By Corollary \ref{coro3}, any branch of solutions $x=x(q)$ of
$P(x)=R(x)-q=0$ passing by $(x,q)=(0,0)$  satisfies the differential
equation
\begin{equation}\label{eq:edobis}
x'=\frac {R'(x)U(x)}{D(q)},
\end{equation}
where $R'(x)U(x)$ is a polynomial in $x$ of degree  $(n-1)^2$ and
$D(q)$ is a polynomial in~$q$ of degree  $n-1.$ By dividing $R'U$ by
$P$ we get that $R'(x)U(x)=P(x)Q(x)+W(x),$ where $W$ is a polynomial
in $q$ and $x$ of degree at most $n-1$ in this last variable. That
is,
\[
W(x)=\sum_{j=0}^{n-1} w_j(q) x^j,
\]
where the functions $w_j(q)$ are polynomials in $q.$ Hence, since
$P(x(q))\equiv0,$ when $x=x(q)$ it holds that
\begin{equation}\label{eq:edofinal}
x'=\frac {R'(x)U(x)}{D(q)}=\frac {P(x)Q(x)+W(x)}{D(q)}=\frac
{W(x)}{D(q)}=\sum_{j=0}^{n-1} \frac{w_j(q)}{D(q)}
x^j=\sum_{j=0}^{n-1} a_j(q) x^j,
\end{equation}
as we wanted to prove.

Let us detail the corresponding ODE \eqref{eq:edofinal} for
$n=2,3,4.$

For the quadratic equation $P(x)=x^2+px-q=0,$ we have that
\[
D(q)=p^2+4q\quad\mbox{and}\quad  U(x)=1.
\]
Hence $x(q)$ satisfies \eqref{eq:edobis},
\begin{equation}\label{eq:n=1gen}
x'=\frac {R'(x)U(x)}{D(q)}= \frac{2x+p}{p^2+4q}= \frac{2}{p^2+4q} x+
\frac{p}{p^2+4q},
\end{equation}
that is already a linear ODE. Its general solution is
\[
x(q)=\frac{-p+K \sqrt{p^2+4q\,}}2.
\]
By imposing the initial condition $x(0)=0$ we arrive again to the
babylonian solution
\[
x(q)=\frac{-p+\operatorname{sgn}(p)\sqrt{p^2+4q\,}}2.
\]

For the cubic equation $P(x)=x^3+px-q=0,$
\[
D(q)=-(4p^3+27q^2)\quad\mbox{and}\quad  U(x)=-(3x^2+4p).
\]
Since  $R'(x)U(x)=-(9x^4+15px^2+4p^2)=-9xP(x)-(6px^2+9qx+4p^2)$ and
$P(x(q))\equiv0,$ we have that $x=x(q)$ satisfies the Riccati
equation
\begin{equation}\label{eq:cubic}
x'= \frac{6px^2+9qx+4p^2}{4p^3+27q^2}=
\frac{6p}{4p^3+27q^2}x^2+\frac{9q}{4p^3+27q^2}x+\frac{4p^2}{4p^3+27q^2}.
\end{equation}

Finally, we consider the quartic equation $P(x)=x^4+px-q=0.$ Here we have
\[
D(q)=-(27p^4+256q^3)\quad\mbox{and}\quad U(x)=
-(16x^6+40px^3+27p^2)
\]
and
\begin{align*}
R'(x)U(x)&=-(64x^9+176px^6+148p^2x^3-27p^3)\\&=
-\big(64x^5+112px^2+64qx\big)P(x)-\big(36{p}^{2}{x}^{3}+48pq
{x}^{2}+64q^2 x+27{p}^{3}\big).
\end{align*}
Hence $x(q)$ satisfies the Abel ODE
\[
x'=\frac{36{p}^{2}}{27p^4+256q^3}x^3+\frac{48pq}{27p^4+256q^3}x^2
+\frac{64q^2}{27p^4+256q^3}x+\frac{27{p}^{3}}{27p^4+256q^3}.
\]

\section{Polynomials and linear ODE: proof of Theorem \ref{teo3}}\label{se:order-n}

We prove Theorem \ref{teo3} and we apply it to the low degree cases. We
recover again Cardano's formula, obtaining it as a particular
solution of the equation for the harmonic oscillator. We also get the solution of
quartic equations in terms of a generalized hypergeometric function that gives
an alternative expression to the classical algebraic one, also presented in
Section \ref{se:cuartic} of the Appendix.

We start proving the following simple lemma.

\begin{lemma}\label{le:ode-n} Let $x(q)$ be a solution of  the polynomial equation $P(x)=R(x)-q=0$ and consider
$v(q)={A(x(q),q)}/{D^m(q)},$ where $0<m\in\N,$
$D(q)=\operatorname{dis}_x (P(x))$ and  $A$ is a polynomial. Then
$v'(q)={B(x(q),q)}/{D^{m+1}(q)},$ for some new polynomial $B(x,q).$
\end{lemma}

\begin{proof} We have that
\begin{align*}
v'(q)=\frac{\dfrac{\partial A(x(q),q)}{\partial x}x'(q)
+\dfrac{\partial A(x(q),q)}{\partial
q}}{{D^m(q)}}-m\frac{A(x(q),q)D'(q)}{D^{m+1}(q)}=
\frac{B(x(q),q)}{D^{m+1}(q)},
\end{align*}
where $B(x,q)=\dfrac{\partial A(x,q)}{\partial
x}R'(x)U(x)+\dfrac{\partial A(x,q)}{\partial q}D(q)-mA(x,q)D'(q)$,
we have used  Corollary \ref{coro3} and $U$ is the polynomial
appearing in its statement.
\end{proof}

\begin{proof}[Proof of Theorem \ref{teo3}]
We start as in the proof of Theorem \ref{teo2}, recalling that by
Corollary \ref{coro3} it holds that
\begin{equation*}
x'=\frac {R'(x)U(x)}{D(q)}=:\frac {c_1(x)}{D(q)},
\end{equation*}
where $c_1$ is a polynomial in $x$ of degree  $(n-1)^2.$ Notice that
applying re-iteratively Lemma \ref{le:ode-n}, defining $v=x^{(k)},
k=2,3,\ldots$ we obtain that
\begin{equation*}
x^{(k)}=\frac {C_k(x,q)}{D^k(q)},\quad 1<k\in\N,
\end{equation*}
where $C_k(x,q)$ are polynomials of increasing degrees in $x$, defined
recursively as
\[
C_{k+1}(x,q)=\frac{\partial C_{k}(x,q)}{\partial
x}R'(x)U(x)+\frac{\partial C_{k}(x,q)}{\partial
q}D(q)-kC_{k}(x,q)D'(q),
\]
and $C_1(x,q)=c_1(x).$ As in the proof of Theorem \ref{teo2}, we can
write
\[
C_k(x,q)=Q_k(x,q)P(x)+B_k(x,q),\quad 0<k\in\N,
\]
where each $B_k$ is a polynomial in $q$ and $x,$ of degree at most
$n-1$ in this last variable. Hence, $x=x(q)$ satisfies
\begin{equation}\label{eq:1tok}
x^{(k)}=\frac {B_k(x,q)}{D^k(q)}=\sum_{j=0}^{n-1} b_{k,j}(q) x^j,\quad
0<k\in\N.
\end{equation}
 For $k=1$ this ODE in the one of Abel type
given in Theorem \ref{teo2}.

Let us explain how to obtain a $(n-1)$-th order linear differential
equation by using~\eqref{eq:1tok} for $k=1,2,\ldots,n-1.$ In fact,
as a first step,  we  prove that these $n-1$ ODE can be transformed
into $n-2$ ODE, where their left hand sides are polynomials of
degree one in the variables $x',x'',\ldots, x^{(n-1)}$ and with
coefficients that are rational functions of $q$, while their right
hand sides continue being polynomials in $x$ but have decreased
their degrees to $n-2.$

 If at least $n-2$ of the functions $b_{n-1,j}$
$ j=1,2,\ldots,n-1,$ identically vanish, we are done. Otherwise, at
least two of them, say $b_{n-1,i}$ and $b_{n-1,\ell},$ are not
identically zero. Then, by computing  $
x^{(i)}/b_{n-1,i}(q)-x^{(\ell)}/b_{n-1,\ell}(q)$ we cancel the term
$x^{n-1}$ in the corresponding right hand side, obtaining one of the
new desired relations. Doing the same procedure with several couples
of relations \eqref{eq:1tok} satisfying that $b_{n-1,j}\ne0$ we
obtain the $n-2$ searched relations.

Starting from these new relations, combining them in a similar way,
we obtain, for each $m=3,\ldots,n-1,$ in each step $n-m$ ODE whose
right hand sides have degree $n-m$ in the variable $x.$ The last
step  of this procedure gives the desired linear differential
equation.
\end{proof}

Although the linear ODE given in Theorem \ref{teo3} can be obtained
in general, their expressions are  huge. To show some examples we
give these ODE for the particular case of trinomial polynomials
\begin{equation}\label{eq:tri}
P(x)=x^n+px-q=0,
\end{equation}
when $n=3,4,5,6.$ As we will see, in this case of trinomial polynomials the resulting ODE has a simple expression.

Notice also that is not difficult to see that for $n>1,$ near $q=0,$
the solution of~\eqref{eq:tri}  is
\[
x=x(q)= \frac1p\, q-\frac1{p^{n+1}}\,q^n+ o(q^n).
\]
Hence all the $(n-1)$-th order linear ODE that we will obtain
have to be solved with the initial conditions
\begin{equation}\label{eq:ic}
x(0)=0,\,\, x'(0)=\frac1p,\,\, x''(0)=x'''(0)=\cdots=x^{(n-2)}(0)=0.
\end{equation}

For $n=3,$ we already know  from \eqref{eq:cubic}  that
\begin{equation*}
x'=\frac{6p}{4p^3+27q^2}x^2+\frac{9q}{4p^3+27q^2}x+\frac{4p^2}{4p^3+27q^2}.
\end{equation*}
By using the procedure detailed in the proof of Theorem \ref{teo3}
we obtain that
\begin{equation*}
x''=-\frac{162pq}{(4p^3+27q^2)^2}x^2+\frac{12p^3-162q^2}
{(4p^3+27q^2)^2}x-\frac{108p^2q}{(4p^3+27q^2)^2}.
\end{equation*}
Hence, since $27qx'+(4p^3+27q^2)x''=3x,$ we arrive to the ODE
\begin{equation}\label{eq:n=3}
(4p^3+27q^2)x''+27qx'-3x=0.
\end{equation}

For the non trinomial case the associated ODE is non homogeneous in
general (see Remark \ref{re:nh}). Let us solve it. We introduce a
new independent variable $t,$ as $q=g(t),$ for some smooth function
$g$ invertible at $t=0$ and such that $g(0)=0.$ Then $y(t)=x(g(t))$
is a solution of a new second order linear ODE in $y(t).$
Straightforward computations give that the coefficient of $y''(t)$
for this new ODE is $(4p^3+27g^2(t))/(g'(t))^2.$ Hence, to find a
new simple ODE we impose that
\[
\frac{4p^3+27g^2(t)}{(g'(t))^2}=\operatorname{sgn}(p) 3 .
\]
Solving it we obtain that, when $p\ne0,$ one of its solutions
satisfying $g(0)=0$ and $g'(0)\ne0,$ is
\begin{equation*}
g(t)=\begin{cases} \dfrac{2\sqrt3}9 p\sqrt{p}\,
\operatorname{sinh}\left(3t\right),
&\mbox{when} \quad p>0, \\[0.4cm]
\dfrac{2\sqrt3}9 p\sqrt{-p}\, \sin\left(3t\right), &\mbox{when}
\quad p<0.
\end{cases}
\end{equation*}
In fact, taking these $g$'s we get that \eqref{eq:n=3} is
transformed into the simple equation
\[
y''-\operatorname{sgn}(p)\,y=0.
\]
Hence,  the general solution of \eqref{eq:n=3} is
\begin{equation*}
x(q)=\begin{cases} C_1\operatorname{sinh}\left(g^{-1}(q)\right)
+C_2\operatorname{cosh}\left(g^{-1}(q)\right),
&\mbox{when} \quad p>0, \\[0.4cm]
C_1\sin\left(g^{-1}(q)\right) +C_2\cos\left(g^{-1}(q)\right),
&\mbox{when} \quad p<0,
\end{cases}
\end{equation*}
where $C_1$ and $C_2$ are arbitrary constants. By imposing the initial
conditions \eqref{eq:ic} and computing $g^{-1}(q)$ we get
\begin{equation*}
x(q)=\begin{cases}
\dfrac{2\sqrt{p}}{\sqrt{3}}\operatorname{sinh}\left(
\dfrac13\operatorname{arcsinh} \left(\dfrac32\dfrac{\sqrt3
q}{p\sqrt{p}}\right) \right),
&\mbox{when} \quad p>0, \\[0.4cm]
\dfrac{2\sqrt{-p}}{\sqrt{3}}\sin\left( \dfrac13\arcsin
\left(\dfrac32\dfrac{\sqrt3 q}{p\sqrt{-p}}\right) \right),
&\mbox{when} \quad p<0,
\end{cases}
\end{equation*}
where the second equality takes real values only for
$|q|<\sqrt{-4p^3/27\,}.$ These expressions coincide with the ones
obtained in Section \ref{se:car}, see \eqref{eq:solu} and
\eqref{eq:solu22}, and lead us again to Cardano's formula.

The ODE obtained for $n=4,5$ and $6$ can be obtained similarly. We
skip the details and we only show the final results.

For $n=4,$
\begin{equation}\label{eq:n=4}
(27p^4+256q^3)x'''+1152q^2x''+688qx'-40x=0.
\end{equation}

For $n=5,$
\begin{equation*}\label{eq:n=5}
(256p^5+3125q^4)x''''+31250q^3x'''+73125q^2x''+31875qx'-1155x=0.
\end{equation*}

For $n=6,$
\begin{align*}\label{eq:n=6}
(3125p^6+46656q^5)x'''''&+816480
q^4x''''+4153680q^3x'''\nonumber\\&+6658200q^2x''+2307456qx'-57456x=0.
\end{align*}

\begin{remark}\label{re:nh} For general polynomials (not in trinomial form) it can be seen that
 the linear differential
equations given in Theorem \ref{teo3} are no more homogeneous. This
is already the case for polynomials of degree $2,$ see
\eqref{eq:n=1gen}. As an example we give it for $x^3+sx^2+px-q=0.$
The associated ODE is
\[
\big(4p^3+27q^2+18pqs-p^2s^2-4qs^3) x''+ \big(27q+9ps-2s^3\big)
x'-3x-s=0.
\]
 When $s=0$, the above ODE reduces to \eqref{eq:n=3}.

\end{remark}

\subsection{Again quartic equations}\label{se:444}
 To solve the quartic we have to  find the solution of \eqref{eq:n=4}
with the initial conditions given in \eqref{eq:ic}, that is
$x(0)=x''(0)=0$ and $x'(0)=1/p.$  By using Mathematica we arrive to
\[
x=x_1(q)=_3\!\!F_2\left(\frac14,\frac12,\frac34;\frac23,\frac43;-\frac{256q^3}{27p^4}\right)\frac
q p
\]
and with Maple we obtain
\[
x=x_2(q)=_2\!\!F_1\left(-\frac1{24},\frac5{24};\frac23;-\frac{256q^3}{27p^4}\right)\cdot\phantom{}
_2F_1\left(\frac7{24},\frac{13}{24};\frac43;-\frac{256q^3}{27p^4}\right)\frac
q p,
\]
where $_n F_m(\cdot;\cdot;x)$ are the classical hypergeometric
functions.\\
As an other result for the quartic, by a direct substitution it is easy to check that
\begin{equation*}
x=x_3(q)=\frac1{2w(q)} \left(\sqrt{2pw^3(q) -1\,}-1\right),
\end{equation*}
where $w=w(q)$ satisfies $-p^2w^6+4qw^4+1=0$  and
$w(0)=1/\sqrt[3]{p},$ solves it. Notice that this can be done
algebraically  because there is the algebraic relation between
$x_3(q)$ and $w(q),$  $(2x_3w+1)^2=2pw^3-1,$ and $w(q)$ also
satisfies a bi-cubic algebraic equation. See Section
\ref{se:cuartic}, and in particular  \eqref{eq:clau}, to understand
how we have obtained the expression   $x_3.$

In particular, by the uniqueness of solutions theorem, it holds that
$x_i(q)=x_j(q),$ for all $i,j\in\{1,2,3\},$ although it seems not
easy to prove these equalities without passing by the differential
equation. It is a challenge to extract the expression of the algebraic
solution of \eqref{eq:n=4} by using only the associated ODE.

\section{Appendix}

For completeness we include in this appendix some classical
approaches to solve cubic and quartic equations. While for the cubic
equations there is nothing new, the solutions of the quartic are given
in a form that is not the most commonly used, but that is very
practical and it is also suitable for our approach to the problem.

\subsection{Cubic equations}\label{ss:cla}

The cubic polynomial equations were solved during the XVI Century by
the Italian school and the protagonists  were Scipione del Ferro,
Niccol\`{o} Fontana (Tartaglia) and Gerolamo Cardano.

As usual,  the cubic polynomial equation $
y^3+by^2+cy+d=0,$ is  transformed into the simpler one
\begin{equation}\label{eq:cubica} x^3+px+q=0,
\end{equation}
for some suitable $p$ and $q,$
 by introducing the new variable  $x=y+\frac{b}{3}.$ We also
 consider that $p\ne0,$ because otherwise its solutions can be
 trivially found.

 We will recall two different well-known ways for solving it.
  We start with the most classical one.  We look for a solution of the form~$x=u+v$. Replacing it
in \eqref{eq:cubica} we get $u^3+v^3+q+(3uv+p)(u+v)=0.$ Now we
impose that  $u$ and $v$ simultaneously satisfy  $ u^3+v^3+q=0$ and
$3uv+p=0.$  By isolating $v$ from the second equation  and replacing
it into the first one we get that  $z=u^3$ satisfies
$z^2+qz-{p^3}/{27}=0.$ Solving this second degree equation and using
that $x=u-p/(3u)$ we arrive to the celebrated Cardano's formula,
$$
x=\root{3}\of{-\frac{q}{2}+\sqrt{\frac{q^2}{4}+\frac{p^3}{27}\,}\,}
-\dfrac{p}{3\root{3}\of{-\frac{q}{2}+\sqrt{\frac{q^2}{4}+\frac{p^3}{27}\,}}\,}.
$$

A different approach is due to Fran\c{c}ois Vi\`{e}te (Vieta). His starting
point  is the trigonometric  identity
\begin{equation}\label{eq:trig}
4\cos^3(\theta)-3\cos(\theta)-\cos(3\theta)=0.
\end{equation}
When $p<0,$ we perform in \eqref{eq:cubica} the change of variables
$x=u\cos(\theta),$ with $u=2\sqrt{-p/3\,}\,$ and multiply the
equation by $4/u^3.$ We arrive to \[
4\cos^3(\theta)-3\cos(\theta)-\frac{3q}{2p}\sqrt{\frac{-3}{p}\,}=0.
\]
Hence, by using \eqref{eq:trig}, when
$\left|\frac{3q}{2p}\sqrt{\frac{-3}{p}\,}\right|\le1,$  the three
solutions of the cubic equation can be obtained from
\[
x=2\sqrt{\frac{-p}3\,}\cos\left(\frac13\arccos
\left(\frac{3q}{2p}\sqrt{\frac{-3}{p}\,}\right)\right),
\]
taking the different values of the  $\arccos$ function. When the
inequality does not hold or $p>0,$ it is possible to consider the
extension of the  $\cos$ function to $\C$ or to use that the $\cosh$
function, $\cosh(x)=(\exp(x)+\exp(-x))/2,$ also satisfies
\begin{equation*}
4\cosh^3(\theta)-3\cosh(\theta)-\cosh(3\theta)=0
\end{equation*}
and then use similar ideas to obtain the solutions of the cubic
equation.

\subsection{Quartic equations}\label{se:cuartic}

The quartic equation was solved by Ludovico Ferrari, only some few
years after the solution of the cubic one. Essentially its solution
is based on some tricks for completing squares that strongly use the
solution of the cubic equation. As we have already commented, we
present a simple and practical version of that approach that is also
suitable for our interests.

By a translation, any quartic equation can we  written as
\begin{equation}\label{eq:quar-general}
x^4+cx^2+dx+e=0, \quad d\ne0,
\end{equation}
where we discard the trivial case $d=0,$ because then the equation
can be easily solved.  Trying to get complete squares in both sides
we write it as
\[
x^4+(c+u^2)x^2+\frac{(c+u^2)^2}4= u^2x^2-dx-e+\frac{(c+u^2)^2}4,
\]
for some $u$ to be determined. Therefore it is natural to impose
that
\[
-e+\frac{(c+u^2)^2}4= \frac{d^2}{4u^2} \Longleftrightarrow
Q(u):=u^6+2cu^4+(c^2-4e)u^2-d^2=0.
\]
Since $d\ne0,$  any value $u$ satisfying  the above bi-cubic
equation is non-zero. Hence for any such $u,$ equation
\eqref{eq:quar-general} writes as
\[
\left(x^2+\frac{c+u^2}2\right)^2=\left(ux-\frac d{2u} \right)^2.
\]
Then the solutions of \eqref{eq:quar-general} coincide with the
solutions of the two quadratic equations
\[
x^2+\frac{c+u^2}2=\pm\left(ux-\frac d{2u} \right).
\]
By solving them we obtain that the four solutions of
\eqref{eq:quar-general} are
\[
\frac u 2 \left( 1\pm\sqrt{-\frac{2d}{u^3}
-\frac{2c}{u^2}-1\,}\right),\quad -\frac u 2 \left( 1\pm
\sqrt{\frac{2d}{u^3} -\frac{2c}{u^2}-1\,}\right),
\]
where $u$ is any solution of $Q(u)=0.$ Taking $w=1/u,$ them can also
be written  as
\[
\frac 1{2w} \left( 1\pm\sqrt{-2d w^3 -2c w^2-1\,}\right),\quad
-\frac 1{2w} \left( 1\pm\sqrt{2d w^3 -2c w^2-1\,}\right),
\]
where $w$ is any solution of the bi-cubic equation
\[-d^2w^6+(c^2-4e)w^4+2cw^2+1=0.\]

In particular, for the trinomial quartic equation $x^4+px-q=0,$ the
solution $x(q)$ that tends to $0$ when $q$ also goes to zero is
\begin{equation}\label{eq:clau}
x(q)=\frac1{2w} \left(\sqrt{2pw^3 -1\,}-1\right),
\end{equation}
where $w=w(q)$ satisfies $-p^2w^6+4qw^4+1=0$  and
$w(0)=1/\sqrt[3]{p}.$

In fact,  a simple a posteriori proof that the solution $x^4+px-q=0$
is given in \eqref{eq:clau} can be done as follows: write
\eqref{eq:clau} as
\[
S(x,w):= (2xw+1)^2+1-2pw^3=0
\]
and observe that
\[
\operatorname{Res}_w(S(x,w),-p^2w^6+4qw^4+1)=-4096p^2\big(x^4+px-q\big)^3,
\]
where $\operatorname{Res}_w$ denotes the resultant with respect to
$w,$ see \cite{St}. Hence, when $p\ne0,$ if $w$ satisfies
simultaneously $S(x,w)=0$ and $-p^2w^6+4qw^4+1=0,$ the
corresponding~$x$ is a zero of the quartic polynomial.


\subsection*{Acknowledgements}
The first author is supported by Ministerio de Ciencia, Innovaci\'{o}n y
Universidades of the Spanish Government through grants
MTM2016-77278-P (MINECO/AEI/FEDER, UE) and by  grant 2017-SGR-1617
from AGAUR,  Generalitat de Catalunya.

\end{document}